\documentclass[12pt]{amsart}
\usepackage{amsfonts}
\usepackage{amsfonts, mathdots}
\usepackage{amssymb,latexsym,color}
\usepackage{enumerate}\usepackage{arydshln}
\usepackage{hyperref}
\makeatletter
\@namedef{subjclassname@2010}{%
  \textup{2010} Mathematics Subject Classification}
\makeatother

\newtheorem{thm}{Theorem}[section]
\newtheorem{cor}[thm]{Corollary}
\newtheorem{lem}[thm]{Lemma}

\newtheorem{prop}[thm]{Proposition}

\theoremstyle{definition}

\newtheorem{rem}[thm]{Remark}

\numberwithin{equation}{section}

\def\diag{{\rm diag}}

\begin{document}

\baselineskip=17pt

\title[Determinant inequality]{Extension of a result of Haynsworth and Hartfiel}

\author[M. Lin]{Minghua Lin}
\address{Department of Mathematics and Statistics\\
University of Victoria\\
 Victoria, BC, Canada, V8W 3R4.}
\email{mlin87@ymail.com}

\date{}

 \begin{abstract} About last 70s, Haynsworth \cite{Hay70} used a result of the Schur complement to refine a determinant inequality for positive definite matrices.   Haynsworth's result was improved by  Hartfiel \cite{Har73}. We extend their result to a larger class of matrices, namely, matrices whose numerical range is contained in a sector. Our proof relies on a number of new relations for the Schur complement of this class of matrices.
 \end{abstract}



\maketitle

\section{Introduction}
\label{sec:intro}

 We start with the notation used in this paper. Let $\mathbb{M}_n$ be the set of all $n\times n$ complex matrices. For $A\in \mathbb{M}_n$, the conjugate transpose of $A$ is denoted by $A^*$, the  real and imaginary part of $A$ are in the sense of the Cartesian decomposition and they are denoted by $\Re A= \frac{1}{2} \left(A+ A^*\right)$ and $\Im A= \frac{1}{2i} \left(A - A^*\right)$, respectively. For two Hermitian matrices $A, B\in \mathbb{M}_n$, we write $A\ge B$ (or $B\le A$) to mean that $A-B$ is positive semidefinite.  We also consider $A\in \mathbb{M}_n$ to be partitioned as  \begin{eqnarray}\label{par}
A=\begin{bmatrix} A_{11} & A_{12}\\ A_{21}  &A_{22}\end{bmatrix},\end{eqnarray} where diagonal blocks are square matrices. If $A$ is nonsingular, then we partition $A^{-1}$ conformally as $A$. If $A_{11}$ is nonsingular, then the Schur complement of $A_{11}$ in $A$ is defined by $A/A_{11}=A_{22}-A_{21}A_{11}^{-1}A_{12}$. The term ``Schur complement" and the notation were first brought in by Haynsworth. We refer the readers to \cite{Zhang05} for a survey of this important notion and its far reaching applications in various branches of mathematics.

 Recall that the numerical range (also known as the field of values) of $A\in\mathbb{M}_n$ is defined by
\begin{eqnarray*} W(A) = \{x^*Ax : x\in  \mathbb{C}^n, x^*x = 1\}.
\end{eqnarray*} Also, we define a sector on the complex plane \begin{eqnarray*} S_\alpha=\{z\in \mathbb{C}: \Re z>0, |\Im z| \le (\Re z) \tan \alpha\}, \qquad \alpha\in [0, \pi/2).  \end{eqnarray*}
Clearly, if $A$ is positive definite, then $W(A)\subset S_0$.

For fundamentals of numerical range, see \cite{GR97, HJ91}. As $0\notin S_\alpha$, if $W(A)\subset S_\alpha$, then $A$ is necessarily nonsingular.

The main object we are dealing in this paper is a class of matrices whose numerical range is contained in $S_\alpha$. Part of the motivation for investigating this class of matrices comes from the search for the optimal growth factor in Gaussian elimination; see, for example,  \cite{Dru13a, GIK02, Hig98, LS14, Lin14}.

Let $A, B\in \mathbb{M}_n$ be positive definite. It is well known that
\begin{eqnarray}\label{e12} \det(A+B)\ge \det A+\det B.
\end{eqnarray}

 Haynsworth  proved the following refinement of (\ref{e12}).
 \begin{thm} \label{thm1} \cite[Theorem 3]{Hay70}  Suppose $A, B\in \mathbb{M}_n$ are positive definite.  Let $A_k$ and $B_k$, $k = 1, \ldots,  n-1$, denote the $k$-th principal submatrices of $A$ and $B$ respectively. Then
 \begin{eqnarray}\label{hay}     &&  \det(A+B)\ge\left(1+\sum_{k=1}^{n-1}\frac{\det B_k}{\det A_k}\right)\det A+\left(1+\sum_{k=1}^{n-1}\frac{\det A_k}{\det B_k}\right)\det B.
\end{eqnarray}
\end{thm}

  Hartfiel \cite{Har73} obtained an improvement of  (\ref{hay}): under the same condition as in Theorem \ref{thm1},
   \begin{eqnarray}\label{har}  \begin{aligned} \det(A+B) &\ge\left(1+\sum_{k=1}^{n-1}\frac{\det B_k}{\det A_k}\right) \det A +\left(1+\sum_{k=1}^{n-1}\frac{\det A_k}{\det B_k}\right)\det B \\ & \qquad +(2^n-2n)\sqrt{\det AB}. \end{aligned}
\end{eqnarray}

  Haynsworth's proof of (\ref{hay}) relies on an inequality for the Schur complement \cite[Theorem 2]{Hay70}:  Let $A, B\in  \mathbb{M}_n$ be positive definite and be comformally partitioned as in (\ref{par}). Then
   \begin{eqnarray}\label{schur-pd}  (A+B)/(A_{11}+B_{11})\ge A/A_{11} +B/B_{11}.
\end{eqnarray}

In this paper, we first extend (\ref{schur-pd}), then as an application, we obtain a generalization of (\ref{har}) and so (\ref{hay}).

\section{Preliminaries}
The larger class of matrices dealt in this paper has some nice closure properties, just like the class of positive definite matrices.  For example, the following proposition says that the Schur complement is closed.

\begin{prop} \label{prop1} Let  $A\in \mathbb{M}_n$ be partitioned as in (\ref{par}). If $W(A)\subset S_\alpha$, then $W(A/A_{11})\subset S_\alpha$ as well. \end{prop}
\begin{proof} Clearly, if $W(A)\subset S_\alpha$, then $W(A^*)\subset S_\alpha$ and $W(A_{22})\subset S_\alpha$. Also, for any nonsingular $X\in \mathbb{M}_n$, $W(A)=W(XAX^*)$. Therefore,  $W(A^{-1})=W(AA^{-1}A^*)=W(A^*)\subset S_\alpha$. The desired result follows by observing that $(A/A_{11})^{-1}=(A^{-1})_{22}$.  \end{proof}

 In the remaining of this section, we present a few auxiliary results.

 \begin{lem} \label{lem1}  Let  $A\in \mathbb{M}_n$ with $W(A)\subset S_\alpha$. Then $A$ can be decomposed as $A=XZX^*$
 for some invertible $X\in \mathbb{M}_n$ and $Z=\diag(e^{i\theta_1}, \ldots, e^{i\theta_n})$ with $|\theta_j|\le \alpha$ for all $j$.\end{lem}

 \begin{rem} The decomposition appears first in \cite[Lemma 1.1]{Dru13a}. In \cite{Zha14}, it is shown that the diagonal entries of $Z$ are unique up to permutation.  \end{rem}

 \begin{lem} \label{lem2}  Let  $A\in \mathbb{M}_n$ with $\Re A$ positive definite.  Then $$(\Re A)^{-1}\ge \Re (A^{-1}).$$ \end{lem}
 \begin{proof} By \cite[Lemma 2.1]{Mat92},  $\Re (A^{-1})=\Big(\Re A+(\Im A)(\Re A)^{-1}(\Im A)\Big)^{-1}$. As $(\Im A)(\Re A)^{-1}(\Im A)$ is positive semidefinite, $\Re (A^{-1})\le (\Re A)^{-1}$ follows.  \end{proof}

 \begin{lem} \label{lem3}  Let  $A\in \mathbb{M}_n$ be partitioned as in (\ref{par}). If $\Re A$ is positive definite, then $$\Re (A/A_{11})\ge (\Re A)/(\Re A_{11}).$$ \end{lem}
 \begin{proof} The notation $(\Re A)/(\Re A_{11})$ makes sense as $\Re A_{11}$ is the $(1,1)$ block of $\Re A$. Consider the Cartesian decomposition $A=M+iN$ with $M=\Re A$, $N=\Im A$ being conformally partitioned as $A$. Then we have the following equality relating the Schur complements \cite[Lemma 2.2]{Lin12},
 \begin{eqnarray*}A/A_{11}=M/M_{11}+i(N/N_{11})+Y(M_{11}^{-1}-iN_{11}^{-1})^{-1}Y^*,
\end{eqnarray*} where $Y=M_{21}M_{11}^{-1}-N_{21}N_{11}^{-1}$.

 As $\Re \Big((M_{11}^{-1}-iN_{11}^{-1})^{-1}\Big)$ is positive semidefinite, so is $\Re \Big(Y(M_{11}^{-1}-iN_{11}^{-1})^{-1}Y^*\Big)$. The desired result follows.
   \end{proof}

 \begin{lem} \label{lem4}  Let  $A\in \mathbb{M}_n$ with $W(A)\subset S_\alpha$.   Then $$\sec^{n}(\alpha)\det(\Re A)\ge |\det A|.$$ \end{lem}
 \begin{proof} Consider the decomposition $A=XZX^*$ as in Lemma \ref{lem1}. Then after dividing by $|\det X|^2$, it suffices to show $\sec^{n}(\alpha)\det(\Re Z)\ge 1$. But each diagonal entry of the diagonal matrix $\sec(\alpha)\Re Z$ is no less than one, implying the result.  \end{proof}

  \begin{rem} The above inequality may be regarded as a complement of the  Ostrowski-Taussky inequality (see \cite[p. 510]{HJ13}). With some minor modification in the proof of \cite[Lemma 3.1]{Zha14}, Zhang showed that actually  the eigenvalues of $\sec(\alpha)\Re Z$ weakly log majorize the singular values of $A$.   \end{rem}

 \section{An extension of (\ref{schur-pd})}
 First of all, we remark that a direct extension of (\ref{schur-pd}) is not valid. That is, assuming $A, B\in \mathbb{M}_n$ with $W(A), W(B)\subset S_\alpha$ are comformally partitioned as in (\ref{par}), it does not hold in general that
 \begin{eqnarray}\label{schur-wrongsec} \Re \Big((A+B)/(A_{11}+B_{11})\Big)\ge \Re (A/A_{11}) +\Re (B/B_{11}).
\end{eqnarray}

 To see this, take $B=A^*$, then (\ref{schur-wrongsec}) contradicts Lemma \ref{lem3}.

 The main result of this section is a correct version of (\ref{schur-wrongsec}).

 \begin{thm} \label{main1}  Let $A, B\in \mathbb{M}_n$ with $W(A), W(B)\subset S_\alpha$ be comformally partitioned as in (\ref{par}). Then  \begin{eqnarray*}\label{schur-sec} ~ ~  \sec^2(\alpha)\Re\Big((A+B)/(A_{11}+B_{11})\Big)\ge\Re  (A/A_{11})+\Re (B/B_{11}).
\end{eqnarray*} \end{thm}
  \begin{proof} We prove the following claim first, which may be regarded as a reverse complement of Lemma \ref{lem3}.

~

  {\bf Claim 1.}  $\sec^2(\alpha)(\Re A)/(\Re A_{11}) \ge \Re (A/A_{11})$.

~

  {\it Proof of Claim 1.} We consider the decomposition $A=XZX^*$ as in Lemma \ref{lem1}. We further partition $X$ as a $2$-by-$1$ block matrix $X= \begin{bmatrix}X_1 \\ X_2\end{bmatrix}$. Then $A=\begin{bmatrix} X_1ZX_1^* & X_1ZX_2^*  \\ X_2ZX_1^*  &  X_2ZX_2^*\end{bmatrix}$. Let $Y=(X^*)^{-1}=\begin{bmatrix}Y_1 \\ Y_2\end{bmatrix}$ be comformally partitioned as $X$. Then $A^{-1}=\begin{bmatrix} Y_1Z^{-1}Y_1^* & Y_1Z^{-1}Y_2^*  \\ Y_2Z^{-1}Y_1^*  &  Y_2Z^{-1}Y_2^*\end{bmatrix}$. Clearly, \begin{eqnarray*}\cos^2(\alpha)(\Re Z)^{-1}\le \Re (Z^{-1}),
\end{eqnarray*}
it follows that
 \begin{eqnarray*}\cos^2(\alpha)Y_2(\Re Z)^{-1}Y_2^*\le \Re (Y_2Z^{-1}Y_2^*),
\end{eqnarray*} i.e.,
\begin{eqnarray*}\cos^2(\alpha) \Big((\Re A)^{-1}\Big)_{22}\le  \Re (A^{-1})_{22},
\end{eqnarray*}
or \begin{eqnarray*}\cos^2(\alpha) \Big((\Re A)/(\Re A_{11})\Big)^{-1}\le  \Re \Big((A/A_{11})^{-1}\Big).
\end{eqnarray*}
Taking the inverse on both hand sides yields \begin{eqnarray*}\sec^2(\alpha) \Big((\Re A)/(\Re A_{11})\Big)\ge  \Big(\Re \Big((A/A_{11})^{-1}\Big)\Big)^{-1}\ge \Re (A/A_{11}),
\end{eqnarray*} in which the second inequality is by Lemma  \ref{lem2}. This completes the proof of Claim 1.

To finish the proof of Theorem \ref{main1}, we observe the following chain of inequalities
\begin{eqnarray*}\Re\Big((A+B)/(A_{11}+B_{11})\Big)&\ge& \Re(A+B)/\Re (A_{11}+B_{11}) \quad \text{by Lemma  \ref{lem3}} \\
 &\ge& (\Re A)/(\Re A_{11})+(\Re B)/(\Re B_{11}) ~ \quad \text{by (\ref{schur-pd})} \\ &\ge & \cos^2(\alpha) \Big(\Re(A/A_{11})+\Re (B/B_{11})\Big)  ~ \text{by Claim 1.}
\end{eqnarray*}
 \end{proof}

  \section{An extension of (\ref{har})}
  As an applicaton of Theorem 3.1, we present the following extension of Haynsworth and Hartfiel's result mentioned in the Introduction.

   \begin{thm} \label{main2}  Suppose $A, B\in \mathbb{M}_n$ such that $W(A), W(B)\subset S_\alpha$.  Let $A_k$ and $B_k$, $k = 1, \ldots,  n-1$, denote the $k$-th principal submatrices of $A$ and $B$ respectively. Then
  \begin{eqnarray*}\sec^{3n-2}(\alpha)|\det(A+B)|&\ge&\left(1+\sum_{k=1}^{n-1}\left|\frac{\det B_k}{\det A_k}\right|\right) |\det A| \\&&  +\left(1+\sum_{k=1}^{n-1}\left|\frac{\det A_k}{\det B_k}\right|\right)|\det B|+(2^n-2n)\sqrt{|\det AB|}.
\end{eqnarray*}
\end{thm}
\begin{proof} Clearly, $(A_{k+1}+B_{k+1})/(A_k+B_k)\in \mathbb{C}$, so $$|(A_{k+1}+B_{k+1})/(A_k+B_k)|\ge \Re \Big((A_{k+1}+B_{k+1})/(A_k+B_k)\Big), \quad k=1, \ldots, n-1.$$ Here we set $A_n=A, B_n=B$. By Proposition \ref{prop1}, $W(A_{k+1}/A_k), W(B_{k+1}/B_k)\subset S_\alpha$; then by  Theorem \ref{main1} and Lemma \ref{lem4},
\begin{eqnarray*}\sec^2(\alpha)\Re \Big((A_{k+1}+B_{k+1})/(A_k+B_k)\Big)&\ge&\Re  (A_{k+1}/A_k)+\Re (B_{k+1}/B_k)\\&\ge&\cos(\alpha)\Big(|A_{k+1}/A_k|+|B_{k+1}/B_k|\Big).
\end{eqnarray*} Hence, $$\sec^3(\alpha)|(A_{k+1}+B_{k+1})/(A_k+B_k)|\ge |A_{k+1}/A_k|+|B_{k+1}/B_k|,$$ that is,
\begin{eqnarray}\label{det1} \sec^3(\alpha)\left|\frac{\det(A_{k+1}+B_{k+1})}{\det(A_k+B_k)}\right|\ge \left|\frac{\det A_{k+1}}{\det A_k}\right|+\left|\frac{\det B_{k+1}}{\det B_k}\right|
\end{eqnarray} for $k=1, \ldots, n-1$.

Taking the product for $k$ from $1$ to $n-1$ in (\ref{det1}) yields
$$\sec^{3(n-1)}(\alpha)|\det(A+B)|\ge |A_1+B_1|\prod_{k=1}^{n-1}\left(\left|\frac{\det A_{k+1}}{\det A_k}\right|+\left|\frac{\det B_{k+1}}{\det B_k}\right|\right).$$
As $|A_1+B_1|\ge \cos(\alpha)(|A_1|+|B_1|)$, we therefore arrive at
\begin{eqnarray*}\sec^{3n-2}(\alpha)|\det(A+B)|&\ge& (|A_1|+|B_1|)\prod_{k=1}^{n-1}\left(\left|\frac{\det A_{k+1}}{\det A_k}\right|+\left|\frac{\det B_{k+1}}{\det B_k}\right|\right)\\&=&\prod_{k=1}^{n}\left(\left|\frac{\det A_k}{\det A_{k-1}}\right|+\left|\frac{\det B_k}{\det B_{k-1}}\right|\right),
\end{eqnarray*} where, by convention,  $\det A_0=\det B_0=1$.

The conclusion follows by taking $a_k=|\det A_k|$, $b_k=|\det B_k|$, $k=0, 1, \ldots, n$, in Claim 2.

~

  {\bf Claim 2.}  Let $a_k, b_k>0$, $k=1, \ldots, n$, also let $a_0=b_0=1$. Then
  \begin{eqnarray*} \prod_{k=1}^n\left(\frac{a_k}{a_{k-1}}+\frac{b_k}{b_{k-1}}\right)&\ge& a_n\left(1+\sum_{s+1}^{n-1}\frac{b_s}{a_s}\right)+b_n\left(1+\sum_{s+1}^{n-1}\frac{a_s}{b_s}\right)\\&&+(2^n-2n)\sqrt{a_nb_n}.
\end{eqnarray*}

  {\it Proof of Claim 2.} The present proof is due to O. Kuba. The orginal proof by the author, which is by induction, is considerably longer.
 Let $\mathbb{N}_n=\{1, 2, \ldots, n\}$, and let $\mathcal{P}(\mathbb{N}_n)$ be the set of subsets of $\mathbb{N}_n$. We consider special subsets $(\mathcal{B}_s)_{1\le s\le n}$ and $(\mathcal{B}'_s)_{2\le s\le n}$ defined by $$\mathcal{B}_s=\{1, 2, \ldots, s\}, \quad  \mathcal{B}'_s=\{s, s+1, \ldots, n\}.$$
 Finally we define $\Omega=\{\emptyset\}\cup \{\mathcal{B}_s: 1\le s\le n\}\cup \{\mathcal{B}'_s: 2\le s\le n\}$ and $\Omega'=\mathcal{P}(\mathbb{N}_n)\setminus \Omega$. Note that $|\Omega'|=2^n-2n$, and that each $k\in \mathbb{N}_n$ belongs to exactly $n$ of the subsets of $\Omega$.

 With this notation, for every $x_1, x_2, \ldots, x_n>0$, we infer that ${\displaystyle\prod_{\mathcal{B}\in \Omega}\prod_{k\in\mathcal{B}}x_k=\prod_{k=1}^n x_k^n}$ and so  ${\displaystyle\prod_{\mathcal{B}\in \Omega'}\prod_{k\in \mathcal{B} }x_k=\prod_{k=1}^n x_k^{2^{n-1}-n}}$, moreover,
 \begin{eqnarray*}\prod_{k=1}^n(1+x_k)&=&\sum_{\mathcal{B}\in \mathcal{P}(\mathbb{N}_n)}\prod_{k\in \mathcal{B}}x_k\\&=&\sum_{\mathcal{B}\in \Omega}\prod_{k\in \mathcal{B}}x_k+\sum_{\mathcal{B}\in \Omega'}\prod_{k\in \mathcal{B}}x_k.
\end{eqnarray*}
 But  \begin{eqnarray*}\sum_{\mathcal{B}\in \Omega}\prod_{k\in \mathcal{B}}x_k=1+\sum_{s=1}^nx_1x_2\cdots x_s+\sum_{s=2}^nx_sx_{s+1}\cdots x_n
\end{eqnarray*} and using the arithemtic mean-geometric mean inequality
  \begin{eqnarray*}\sum_{\mathcal{B}\in \Omega'}\prod_{k\in \mathcal{B}}x_k&\ge& |\Omega'|\left(\prod_{\mathcal{B}\in \Omega'}\prod_{k\in \mathcal{B} }x_k\right)^{1/|\Omega'|}\\&=& (2^n-2n)\left(\prod_{k=1}^n x_k^{2^{n-1}-n}\right)^{1/(2^n-2n)}\\&=& (2^n-2n)\sqrt{x_1x_2\cdots x_n}.
\end{eqnarray*}
So we have
$$\prod_{k=1}^n(1+x_k)\ge 1+\sum_{s=1}^nx_1x_2\cdots x_s+\sum_{s=2}^nx_sx_{s+1}\cdots x_n+(2^n-2n)\sqrt{x_1x_2\cdots x_n}.$$
Taking $x_k=\frac{a_{k-1}b_k}{b_{k-1}a_k}$, for $k=1, \ldots, n$, gives
\begin{eqnarray*}\prod_{k=1}^n\left(1+\frac{a_{k-1}b_k}{b_{k-1}a_k}\right)&\ge&1+\sum_{s=1}^n \frac{b_s}{a_s}+\frac{b_n}{a_n}\sum_{s=2}^n\frac{a_{s-1}}{b_{s-1}}+(2^n-2n)\sqrt{b_n/a_n}\\&=&1+\sum_{s=1}^{n-1}\frac{b_s}{a_s} +\frac{b_n}{a_n}\left(1+\sum_{s=1}^{n-1}\frac{a_s}{b_s}\right)+(2^n-2n)\sqrt{b_n/a_n}.
\end{eqnarray*}
Multiplying both sides of the inequality by ${\displaystyle\prod_{k=1}^n\frac{a_k}{a_{k-1}}=a_n}$ yields the desired inequality. This completes the proof of Claim 2. \end{proof}

Apparently, Theorem \ref{main2} reduces to (\ref{har}) when $\alpha=0$. A matrix $A\in \mathbb{M}_n$ is accretive-dissipative if both $\Re A$, $\Im A$ are positive definite (see \cite{GI05}). Note that if $A$ is accretive-dissipative, then $W(e^{-i\pi/4}A)\subset S_{\pi/4}$. Thus, we have the following corollary.

\begin{cor}   Suppose $A, B\in \mathbb{M}_n$ are accretive-dissipative.  Let $A_k$ and $B_k$, $k = 1, \ldots,  n-1$, denote the $k$-th principal submatrices of $A$ and $B$ respectively. Then
  \begin{eqnarray*} 2^{\frac{3}{2}n-1}|\det(A+B)|&\ge&\left(1+\sum_{k=1}^{n-1}\left|\frac{\det B_k}{\det A_k}\right|\right) |\det A| \\&&  +\left(1+\sum_{k=1}^{n-1}\left|\frac{\det A_k}{\det B_k}\right|\right)|\det B|+(2^n-2n)\sqrt{|\det AB|}.
\end{eqnarray*}
\end{cor}

\subsection*{Acknowledgments} {\small The author thanks P. van den Driessche and P. Zhang for some helpful remarks. In particular, the present proof of Claim 2 is due to O. Kuba.  The author is currently a PIMS Postdoctoral Fellow at the University of Victoria. The support of PIMS is gratefully acknowledged}.


\begin{thebibliography}{11}
 \bibitem {Dru13a} S. W. Drury, Fischer determinantal inequalities and Higham's Conjecture, Linear Algebra Appl. 439 (2013) 3129-3133.
\bibitem {GIK02} A. George, Kh. D. Ikramov, A. B. Kucherov, On the growth factor in Gaussian elimination for generalized Higham matrices, Numer. Linear Algebra Appl. 9 (2002) 107-114.
  \bibitem {GI05}   A. George, Kh. D. Ikramov, On the properties of accretive-dissipative matrices, Math. Notes 77 (2005) 767-776.
\bibitem {GR97} K. E. Gustafson, D. K. M. Rao, Numerical Range: The Field of Values of Linear Operators and Matrices, Springer, New York, 1997.
\bibitem{Har73} D. J. Hartfiel, An extension of Haynsworth's determinant inequality, Proc. Amer. Math. Soc.
  41 (1973) 463-465.
\bibitem{Hay70} E. V. Haynsworth, Applications of an inequality for the Schur complement, Proc.
Amer. Math. Soc. 24 (1970) 512-516.
\bibitem{Hig98} N. J. Higham, Factorizing complex symmetric matrices with positive real and imaginary parts, Math. Comp. 67 (1998) 1591-1599.
  \bibitem {HJ91} R. A. Horn, C. R. Johnson, Topics in Matrix Analysis, Cambridge University
Press, 1991.
 \bibitem {HJ13} R. A. Horn, C. R. Johnson, Matrix Analysis, Cambridge University
Press, 2nd ed. 2013.
\bibitem {LS14} C.-K. Li, N. Sze, Determinantal and eigenvalue inequalities for matrices with numerical ranges in
a sector, J. Math. Anal. Appl. 410 (2014) 487-491.
\bibitem{Lin12} M. Lin, Reversed determinantal inequalities for accretive-dissipative matrices, Math. Inequal. Appl. 12 (2012) 955-958.
\bibitem{Lin14} M. Lin, A note on the growth factor in Gaussian elimination for accretive-dissipative matrices, Calcolo   51 (2014) 363-366.
 \bibitem {Mat92} R. Mathias, Matrices with positive definite Hermitian part: Inequalities and linear systems, SIAM J. Matrix Anal. Appl. 13 (1992) 640-654.
\bibitem{Zhang05}  F. Zhang, The Schur complement and its applications, Springer, New York, 2005.
\bibitem{Zha14} F. Zhang, A matrix decomposition and its applications, Linear Multilinear Algebra (2014) to appear.
\end{thebibliography}
\end{document}